\documentclass[11pt,letterpaper]{article}
\usepackage{graphicx} 
\usepackage[margin=1in]{geometry}

\usepackage{amsmath}
\usepackage{amssymb}
\usepackage{amsthm}
\usepackage{hyperref}
\usepackage{xcolor}
\usepackage{epsfig}
\usepackage{enumitem}

\DeclareMathOperator{\supp}{supp}

\usepackage{enumitem}
\setlist{leftmargin=6mm,nolistsep,noitemsep}

\usepackage[capitalize,noabbrev]{cleveref}
\crefname{question}{Question}{Questions}
\crefname{step}{Step}{Steps}
\crefname{claim}{Claim}{Claims}
\crefname{problem}{Problem}{Problems}
\crefname{definition}{Definition}{Definitions}
\crefname{observation}{Observation}{Observations}

\DeclareMathOperator{\conv}{conv}

\newcommand{\QP}{\text{QP}}

\def\tw{{\rm ptw}}
\def\itw{{\rm itw}}
\def\poly{{\rm poly}}

\newcommand{\ie}{i.e., }

\def\R{{\mathbb R}}
\def\Z{{\mathbb Z}}
\def\Q{{\mathbb Q}}

\newtheorem{theorem}{Theorem}

\newtheorem{proposition}{Proposition}

\newtheorem{lemma}{Lemma}

\theoremstyle{remark}
\newtheorem{remark}{Remark}

\newcommand{\C}{\mathcal C}

\newcommand{\NP}{\mathcal {NP}}
\renewcommand{\L}{\mathcal L}

\usepackage[title]{appendix}

\title{A polynomial-time solvable class of sparse box-constrained polynomial optimization problems
\thanks{The author was supported in part by AFOSR grant FA9550-23-1-0123 and ONR grant N00014-25-1-2491.}}

\author{
Aida Khajavirad
\thanks{Department of Industrial and Systems Engineering,
             Lehigh University.
             E-mail: {\tt aida@lehigh.edu}.
             }
}

\begin{document}

\maketitle

\begin{abstract}
We study the problem of minimizing a multivariate polynomial function over the unit hypercube. Exploiting sparsity in the interaction graph or hypergraph, we identify variables that can be restricted to binary values and eliminate the remaining continuous variables component-wise, reducing the problem to structured binary polynomial optimization. For quadratic objectives, we obtain polynomial-time solvability in the standard Turing model under conditions on treewidth and a nonconvexity measure of the continuous components. For higher-degree objectives, we obtain a polynomial-time algorithm in the unit-cost algebraic RAM model under logarithmic incidence treewidth and small, weakly coupled continuous components.
\end{abstract}


\section{Introduction}
We consider the problem of minimizing a multivariate polynomial function of degree $d \geq 2$ over the unit hypercube. Let $x \in \R^n$ and $\alpha \in \Z^n_{\geq 0}$. Define $x^\alpha := x^{\alpha_1}_1 \ldots x^{\alpha_n}_n$ and $|\alpha| := \sum_{i=1}^n{\alpha_i}$. We define a \emph{box-constrained polynomial optimization problem} of degree $d$ as follows:
\begin{align}\label{polyopt}
    {\rm min} \; &\sum_{|\alpha| \leq d}{c_{\alpha} x^{\alpha}}\\
    {\rm s.t.} \;\; & x \in [0,1]^n, \nonumber
\end{align}
where $c_{\alpha} \in \Q$ for all $\alpha \in \Z^n_{\geq 0}$ with $|\alpha| \leq d$, and we assume that $c_{\alpha} \neq 0$ for at least one $\alpha$ with $|\alpha| = d$.
Throughout the paper, we assume that the input polynomial is given explicitly in the standard monomial basis, that is, as a list of pairs $(\alpha, c_\alpha)$ with $c_\alpha \neq 0$. The input length $\L$ denotes the total number of bits required to encode the nonzero coefficients and their corresponding exponent vectors. 
We assume that the degree $d$ is polynomially bounded in $\L$; that is, $d \in O(\L^c)$ for some fixed constant $c > 0$. In particular, $d$ is not fixed and may grow with the instance.
In the important special case with $d=2$, we obtain the
\emph{box-constrained quadratic optimization problem}:
\begin{align}\label{pQP}
\min \quad & x^\top Q x + c^\top x \\
{\rm s.t.} \quad & x \in [0,1]^n, \nonumber
\end{align}
where $c \in \Q^n$ and $Q \in \Q^{n\times n}$ is a symmetric matrix. 
Problem~\eqref{pQP} is $\NP$-hard in general, as it contains the maximum cut problem as a special case. If $n$ is fixed, then Problem~\eqref{pQP} can be solved in strongly polynomial time using a simple face enumeration argument \cite{vavasis90}.
If Problem~\eqref{pQP} is convex, then thanks to the ellipsoid method,  this problem can be solved in time polynomial in $\L$~\cite{TarKha79}. 
In~\cite{hladik21}, the authors prove that if the rank of $Q$ is fixed, then Problem~\eqref{pQP} can be solved in polynomial time.
If $d \geq 3$ and the polynomial function is convex over the unit hypercube, then an $\epsilon$-optimal solution can be computed in time polynomial in $\L$ and $\log(1/\epsilon)$ using the ellipsoid method~\cite{GroLovSch12} or interior point methods~\cite{NesNem94}. 
For $d\geq 3$, if the polynomial is nonconvex, polynomial-time solvability of the general problem is known only in fixed dimension. Namely, if $n$ is fixed, then Problem~\eqref{polyopt} is solvable in time polynomial in $\L$ via quantifier-elimination techniques~\cite{renegar92}.

\paragraph{Binary polynomial optimization.} In contrast to the continuous setting, the complexity of binary polynomial optimization, \ie the problem of minimizing a multivariate polynomial over the set of 
binary points, has been extensively studied in the literature~\cite{craHanJau90,BorHam02,dPKha23,dPDiG23ALG,dpAk24,CapDelDiG23}. To formally define this problem, we use the hypergraph representation scheme of~\cite{dPKha17MOR}. A hypergraph $G$ is a pair $(V,E)$, where $V$ is a finite set of nodes and $E$ is a set of subsets of $V$ of cardinality at least two, called the edges of $G$. Define $d:= \max_{e \in E} |e|$.
With any hypergraph $G= (V,E)$ and the cost vector $c \in \Q^{V \cup E}$, we associate the following \emph{binary polynomial optimization problem of degree $d$}:
\begin{align}
\label{prob BPO}
\min \quad & \sum_{e\in E} {c_e \prod_{i\in e} {z_i}}+\sum_{i\in V} {c_i z_i} \\
{\rm s.t. } \quad & z \in \{0,1\}^V,\nonumber
\end{align}
where, without loss of generality, we assume that each node is contained in at least one edge, and $c_e \neq 0$ for all $e \in E$. Problem~\eqref{prob BPO} is $\NP$-hard even for $d=2$. 
If the objective function of Problem~\eqref{prob BPO} is submodular, then this problem can be solved in strongly polynomial time~\cite{schrijver2000}. 
Given a hypergraph $G=(V,E)$, the \emph{intersection graph of $G$} is the graph with node set $V$, and where two nodes
$i,j \in V$ are adjacent if $i,j \in e$ for some $e \in E$. Henceforth, we refer to the treewidth of the intersection graph of a hypergraph $G$ as the \emph{primal treewidth} of $G$ and denote it by $\tw(G)$. In~\cite{craHanJau90}, the authors prove that if $\tw(G)$ is logarithmically bounded, then Problem~\eqref{prob BPO} can be solved in strongly polynomial time. Given a hypergraph $G=(V,E)$, the~\emph{incidence graph of $G$}
is a bipartite graph whose vertex set is $V \cup E$ and the edge set is $\{\{i, e\} : i \in V, \ e \in E, \ i \in e\}$. Henceforth, we refer to the treewidth of the incidence graph of a hypergraph $G$ as the \emph{incidence treewidth} of $G$ and denote it by $\itw(G)$.
For any hypergraph $G$, we have  $\itw(G) \leq \tw(G)+1$.
In~\cite{CapDelDiG23}, the authors prove that if $\itw(G)$ is logarithmically bounded, then Problem~\eqref{prob BPO} can be solved in strongly polynomial time. 
Finally,  in~\cite{dPDiG23ALG}, the authors prove that if the hypergraph $G$ is $\beta$-acyclic, then Problem~\eqref{prob BPO} can be solved in strongly polynomial time. 

Let us consider the important special case of Problem~\eqref{prob BPO} with $d=2$. With any graph $G=(V,E)$ and cost vector $c \in \Q^{V\cup E}$, we associate the \emph{binary quadratic optimization problem}: 
\begin{align}
\label{prob QPO}
\min \quad & \sum_{\{i,j\}\in E} {c_{ij} z_i z_j}+\sum_{i\in V} {c_i z_i} \\
{\rm s.t. } \quad & z \in \{0,1\}^V.\nonumber
\end{align}
First, for a graph $G$ we have $\tw(G) = \itw(G)$. Second, for  a graph, the notion of $\beta$-acyclicity in hypergraphs coincides with the usual notion of graph acyclicity.  
Therefore, for Problem~\eqref{prob QPO} the most general sufficient condition for polynomial-time solvability in terms of treewidth is as follows: if the graph $G$ has a logarithmically bounded treewidth, then Problem~\eqref{prob QPO} can be solved in strongly polynomial time. In fact, the results of~\cite{ChaSreHar08,dpAk24} imply that under certain complexity assumptions, a logarithmically bounded treewidth is essentially a necessary and sufficient condition for polynomial-time solvability of Problem~\eqref{prob QPO}.

\paragraph{Our contributions.} In this paper, by building on existing algorithms for binary polynomial optimization and box-constrained polynomial optimization, we obtain new sufficient conditions for the exact solution of Problem~\eqref{polyopt} and Problem~\eqref{pQP} in polynomial time.
For quadratic objectives, our algorithm runs in time polynomial in $\L$ in the standard Turing bit model. For higher-degree objectives, we obtain a polynomial-time algorithm in the unit-cost algebraic RAM model.
To exploit the sparsity of these optimization problems, we model them using graphs or hypergraphs.
Consider Problem~\eqref{polyopt};
we define the \emph{interaction hypergraph} $G=(V,E)$ for this problem as follows. For each variable $x_i$, $i\in [n]:=\{1,\cdots,n\}$, we define a node $i$. For each monomial $c_\alpha x^\alpha$ with $c_\alpha \neq 0$ and $|\supp(\alpha)| \geq 2$, we include the edge $\supp(\alpha):=\{i \in V : \alpha_i \neq 0\}$ in $E$. In the special case with $d=2$, $G$ is a graph and we call it the \emph{interaction graph} of Problem~\eqref{pQP}.

First, we focus on the quadratic case; namely Problem~\eqref{pQP}. 
We partition the variables according to the sign of the diagonal entries of $Q$: variables $x_i$ with $q_{ii}\leq0$ can be restricted to binary values in an optimal solution. Henceforth, we refer to these variables as \emph{hidden binary variables}, while the remaining variables form the continuous part. We decompose the continuous variables into connected components of the interaction subgraph induced by the continuous variables, and eliminate these components independently. 
We show that the complexity of this elimination can be controlled by a parameter measuring the nonconvexity of each continuous component. This permits continuous components of arbitrarily large size when their degree of nonconvexity is sufficiently small. Combining this condition with a logarithmic-treewidth condition on the resulting interaction graph yields a polynomial-time algorithm in the standard Turing bit model  (see Theorem~\ref{th: quad}).
Next, we extend this approach to higher-degree box-constrained polynomial optimization problems. We again separate variables into a hidden binary part and a continuous part and decompose the continuous variables into connected components of the interaction subhypergraph induced by continuous variables. When each continuous component has constant size and directly interacts with only a logarithmic number of binary variables, and when the incidence treewidth of the residual interaction hypergraph is logarithmically bounded, we can eliminate each continuous component locally using polynomial-time algorithms for fixed-dimensional polynomial optimization from real algebraic geometry. Unlike in the quadratic case, the optimal values obtained from these local continuous problems need not be rational; in general, they are real algebraic numbers. We therefore formulate the higher-degree result in the unit-cost algebraic RAM model. Using existing algorithms for binary polynomial optimization and exact sign determination of real algebraic expressions, we show that the resulting problem can be solved in polynomial time in unit-cost algebraic RAM model (see Theorem~\ref{th: polyopt}).

Related to our use of sparsity, Bienstock and Muñoz~\cite{BieMun18} show that box-constrained polynomial optimization problems with bounded primal treewidth admit polynomial-size linear programming approximations that, for any prescribed tolerance, yield a feasible point whose objective value is additively close to the global optimum, whereas our focus is on sufficient conditions for exact polynomial-time solvability.


\section{Box-constrained quadratic optimization}
\label{sec: quad}

In this section, we consider the box-constrained quadratic optimization problem defined by Problem~\eqref{pQP}. Henceforth, given a graph $G=(V,E)$, we denote by ${\rm tw}(G)$ the treewidth of $G$. For any $S \subseteq V$, we denote by $G_S$ the subgraph of $G$ induced by $S$. Consider a node $v \in V$. We define the \emph{neighborhood} of $v$ in $G$ as
\begin{equation}\label{nghoodnode}
N_G(v) := \big\{u \in V : \{u,v\} \in E \big\}.
\end{equation}
Similarly, for a subset $S \subseteq V$, we define the neighborhood of $S$ in $G$ as: 
\begin{equation}\label{nghood}
N_G(S) := \Big(\bigcup_{v\in S} N_G(v)\Big) \setminus S.
\end{equation}

\paragraph{The computational model.}
Throughout this section, running times are measured in the standard Turing bit model. All numerical input data are rational and encoded in binary, and $\L$ denotes their total binary encoding length. Thus, whenever we refer to
polynomial time in this section, we mean time polynomial in $\L$, accounting for the bit complexity of all arithmetic operations.

\subsection{Statement of main result}
We are now ready to state our sufficient conditions for polynomial-time solvability of Problem~\eqref{pQP}.

\begin{theorem}\label{th: quad}
Let $G=(V,E)$ be the interaction graph of Problem~\eqref{pQP}. Define $V^+ := \{i \in V: q_{ii} > 0\}$.
Denote by $\C$ the set of connected components of $G_{V^+}$. Let $Q_C$ denote the principal submatrix of $Q$ indexed by $C$, and let $\kappa_C$ be the smallest integer in $\{0,\ldots,|C|-1\}$ such that every $(|C|-\kappa_C)\times(|C|-\kappa_C)$ principal submatrix of $Q_C$ is positive semidefinite. Denote by
$\bar G$ the graph on $V \setminus V^+$ obtained from $G_{V \setminus V^+}$ by making each $N(C)$ a clique.
Suppose that the following conditions are satisfied:
\medskip
\begin{itemize}[leftmargin=1.0cm]
\item[(i)] ${\rm tw}(\bar G)\in O(\log |V|)$;
\item[(ii)]
$\kappa_C\log\left(1+\frac{|C|}{\kappa_C}\right)
\in O(\log |V|)$ for all $C\in\C$,
where we define $0 \log(1+\frac{|C|}{0}) :=0$.
\end{itemize}
\medskip
Then Problem~\eqref{pQP} can be solved in time polynomial in the input length $\L$ in the Turing bit model.
\end{theorem}

\begin{remark}
Let us consider the two assumptions in Theorem~\ref{th: quad}; these assumptions control two different sources of complexity in Problem~\eqref{pQP}. Namely, assumption~(i) is concerned with the difficulty of the combinatorial component of Problem~\eqref{pQP}, while assumption~(ii) is concerned with the difficulty of its continuous component. 
Since for each $C \in \C$, $N(C)$ is a clique of the graph $\bar G$,
assumption~(i) implies that $|N(C)| \in O(\log |V|)$ for all $C \in \C$. 
On the other hand, assumption~(ii) has a natural interpretation depending on the size of $\kappa_C$ relative to $|C|$. In the worst case where  $\kappa_C=\Theta(|C|)$, we have
$
\kappa_C\log\big(1+\frac{|C|}{\kappa_C}\big)
=\Theta(|C|)
$,
and hence assumption~(ii) becomes
$|C|\in O(\log |V|)$.
On the other hand, if $\kappa_C=\Theta(1)$, then
$\kappa_C\log\big(1+\frac{|C|}{\kappa_C}\big)
=\Theta(\log |C|).
$
Since $|C|\leq |V|$, this term is automatically in $O(\log |V|)$. Therefore, in this regime, assumption~(ii)  places no restriction on the size of $C$ and is in fact redundant. 
\end{remark}

\begin{remark}
The structural quantities appearing in the assumptions of Theorem~\ref{th: quad} can be computed in time polynomial in $\L$ for instances satisfying these assumptions. The interaction graph $G=(V,E)$ can be constructed in polynomial time by inspecting the nonzero entries of the matrix $Q$. The set $V^+$ can be identified by checking the signs of the diagonal entries $q_{ii}$. The connected components of the induced subgraph $G_{V^+}$ and the corresponding neighborhoods $N(C)$ can be computed using standard graph traversal algorithms in time $O(|V|+|E|)$.
For each $C\in\mathcal{C}$, the parameter $\kappa_C$ can be determined by considering, for $r=0,1,\ldots,|C|-1$, all principal submatrices of $Q_C$ of order $|C|-r$ and finding the smallest $r$ for which all such submatrices are positive semidefinite. If $\kappa_C=0$, only the matrix $Q_C$ needs to be checked. Now suppose that $\kappa_C\geq1$. The total number of principal submatrices that need to be considered is at most
$$
\sum_{r=0}^{\kappa_C}\binom{|C|}{r}
\leq
(\kappa_C+1)\Big(\frac{e|C|}{\kappa_C}\Big)^{\kappa_C}.
$$
Since $\kappa_C\leq |C|$, we have
$
\kappa_C+1 \leq\Big(1+\frac{|C|}{\kappa_C}\Big)^{\kappa_C}
$
and
$
\frac{e|C|}{\kappa_C}
\leq
\Big(1+\frac{|C|}{\kappa_C}\Big)^2.
$
Therefore,
$$
\sum_{r=0}^{\kappa_C}\binom{|C|}{r}
\leq
\Big(1+\frac{|C|}{\kappa_C}\Big)^{3\kappa_C}
=
|V|^{O(1)},
$$
where the last equality follows from assumption~(ii). Positive semidefiniteness of each principal submatrix can be checked in polynomial time. Hence, under assumption~(ii), $\kappa_C$ can be computed in polynomial time. The graph $\bar G$ can be constructed in polynomial time from $G$ and the sets $N(C)$.
Moreover, although computing treewidth exactly is $\NP$-hard, there exists an algorithm that, given a graph $\bar G$ and an integer $k$, runs in time $2^{O(k)}|V|$ and either certifies that ${\rm tw}(\bar G)>k$ or constructs a tree decomposition of width $O(k)$~\cite{BodEtAl16}. Therefore, under assumption~(i), a tree decomposition of $\bar G$ of width $O(\log |V|)$ can be constructed in polynomial time.
\end{remark}

\begin{remark}\label{rem:example}
Assumptions~(i)-(ii) of Theorem~\ref{th: quad}
do not imply any nontrivial upper bound on the number of variables corresponding to $V^+$. In fact, it is possible to have both $|V^+| = \Theta(|V|)$ and $|V^-| = \Theta(|V|)$, where we define $V^- := V \setminus V^+$.
To see this, for each integer $m\geq1$, define a graph $G=(V,E)$ as follows. Let
$V^+:=\{p_1,\ldots,p_m\}$, $V^-:=\{z_1,\ldots,z_{m+1}\}$,
and let $V:=V^+\cup V^-$. The nodes in $V^-$ form a path, and each node $p_i\in V^+$ is adjacent to two consecutive nodes $z_i,z_{i+1}\in V^-$. More precisely,
$E:=\big\{\{z_i,z_{i+1}\}:i\in [m]\big\}
\cup \big\{\{p_i,z_i\},\{p_i,z_{i+1}\}:i\in [m]\big\}$.
Choose the diagonal coefficients of $Q$ so that $q_{p_ip_i}>0$ for every $i\in[m]$ and $q_{z_jz_j}\leq0$ for every $j\in[m+1]$, and choose the off-diagonal coefficients so that the interaction graph of $Q$ is $G$. 
Since there are no edges between distinct nodes in $V^+$, the connected components of $G_{V^+}$ are the singletons
$C_i:=\{p_i\}$ for all $i \in [m]$.
For each $i\in[m]$, we have $N(C_i)=\{z_i,z_{i+1}\}$.
Moreover, since $C_i$ is a singleton and $q_{p_ip_i}>0$, the matrix $Q_{C_i}$ is positive definite, and hence $\kappa_{C_i}=0$. Therefore, assumption~(ii) is satisfied.
The subgraph $G_{V^-}$ is the path
$z_1-z_2-\cdots-z_{m+1}$. Furthermore, each set $N(C_i)=\{z_i,z_{i+1}\}$ already forms a clique in $G_{V^-}$. Consequently, making each $N(C_i)$ a clique adds no new edges, and hence
$\bar G=G_{V^-}$. It follows that
${\rm tw}(\bar G)=1$, so assumption~(i) is also satisfied. Finally, $|V|=2m+1$, $|V^+|=m$, $|V^-|=m+1$,
and therefore $|V^+|=\Theta(|V|)$ and $|V^-|=\Theta(|V|)$.
Furthermore, the principal submatrix $Q_{V^+}$ is diagonal with
strictly positive diagonal entries. Therefore,
${\rm rank}(Q)
\geq {\rm rank}(Q_{V^+})=m=\frac{|V|-1}{2}$,
and hence ${\rm rank}(Q)=\Theta(|V|)$.
Thus, this family is covered by Theorem~\ref{th: quad} but is
not covered by either fixed-dimensional tractability results
or the fixed-rank tractability criterion of~\cite{hladik21}.
See Figure~\ref{fig:easyEx}. 
\end{remark}

\begin{figure}[ht]
\centering
\includegraphics[width=0.4\textwidth]{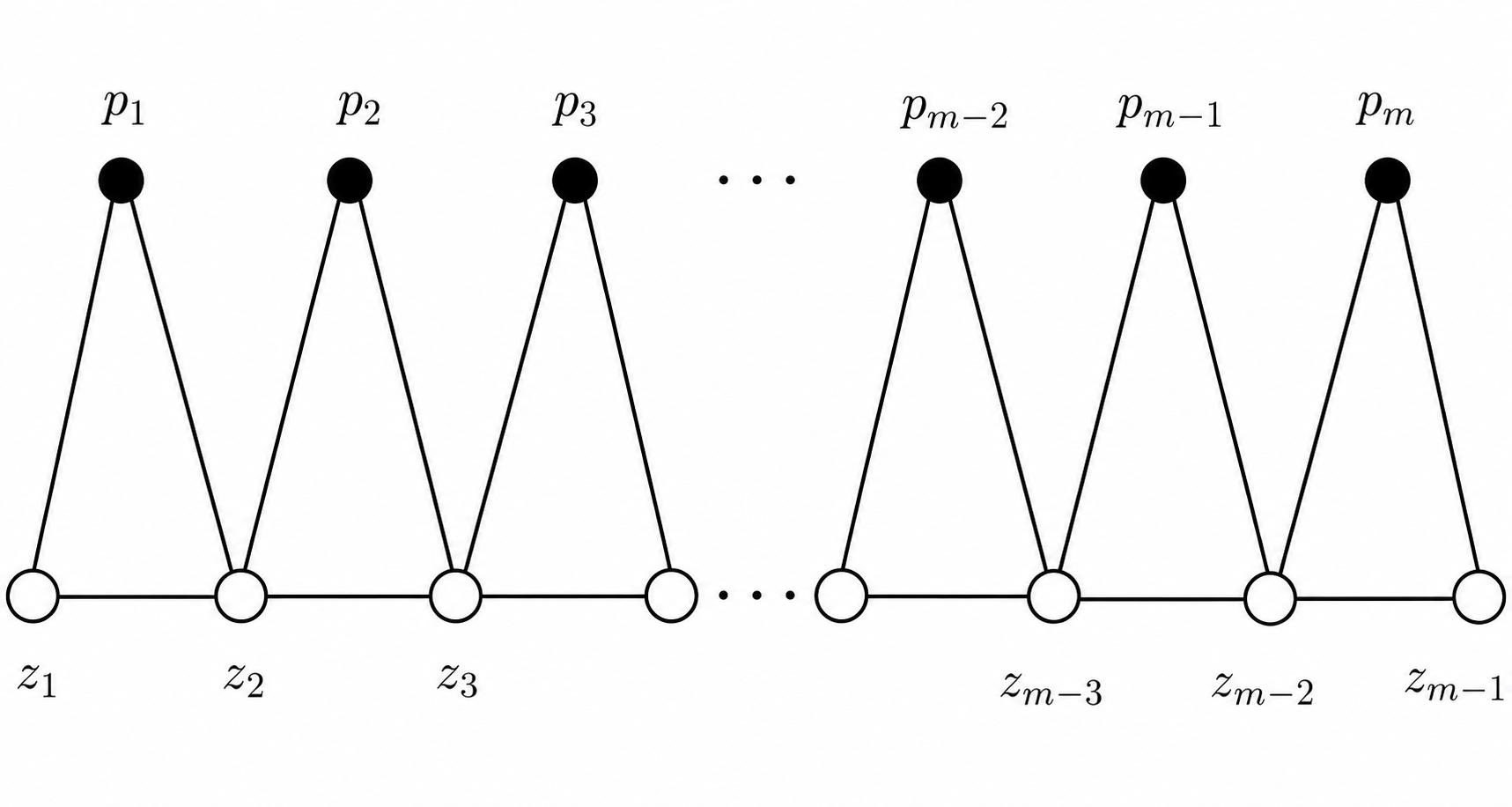}
\caption{
The example of Remark~\ref{rem:example} illustrating the application of Theorem~\ref{th: quad}.
The black nodes represent $V^+$, while the white nodes represent $V^-$
}
\label{fig:easyEx}
\end{figure}

\begin{remark}
    In~\cite{DeyIda25,Aida26}, the authors propose new convex relaxations for Problem~\eqref{pQP}. To this end, they associate a graph $G=(V,E,L)$ with this problem, where $V$ and $E$ are node set and edge set, respectively, of $G$ as defined for our interaction graph, and $L$ denotes the loop set, where $\{i,i\} \in L$, if $q_{ii} \neq 0$. They further define $L=L^- \cup L^+$, where $L^-$ (resp. $L^+$) contain all loops with $q_{ii} < 0$ (resp. $q_{ii} > 0$).
    Subsequently, they study the facial structure of the set:
\begin{align*}
\QP(G) := \conv\Big\{z \in \R^{V\cup E \cup L}: \; & z_{ii} \geq z^2_i, \; \forall \{i, i\} \in L^+, \; z_{ii} \leq z^2_i, \; \forall \{i, i\} \in L^-, \; z_{ij} = z_i z_j, 
\forall \{i,j\} \in E\\ &  z_i \in [0,1], \forall i \in V\Big\},
\end{align*}
where $\conv(\cdot)$ denotes the convex hull. They obtain sufficient conditions under which $\QP(G)$ admits a polynomial-size second-order cone (SOC) or semidefinite programming (SDP) representable formulation that can be constructed in polynomial time. We next summarize these sufficient conditions. Suppose that ${\rm tw}(\bar G) \in O(\log |V|)$, where the graph $\bar G$ is as defined in the statement of Theorem~\ref{th: quad}. We have the following cases:
\medskip
\begin{enumerate}
    \item if $|C| =1$ for all $C \in \C$, then $\QP(G)$ admits a polynomial-size SOC-representable formulation that can be constructed in polynomial time (see Theorem~3 in~\cite{DeyIda25} and Corollary~6 in~\cite{Aida26}).

\item if $|C| =2$ for all $C \in \C$, then $\QP(G)$ admits a polynomial-size SDP-representable formulation that can be constructed in polynomial time (see Theorem~7 in~\cite{Aida26}).
    \end{enumerate}
\medskip
Clearly, conditions~(1) and~(2) are special cases of 
Theorem~\ref{th: quad}. 
It is important to note that the tightness of a polynomial-size SOCP or SDP relaxation of Problem~\eqref{pQP} does not imply its polynomial-time solvability because when the optimal solution of Problem~\eqref{pQP} is not unique,  the solution returned by the SOCP or SDP solver may not be feasible for the nonconvex lifted problem. Moreover, Theorem~\ref{th: quad} does not imply the results of~\cite{DeyIda25,Aida26} either, because polynomial-time solvability of Problem~\eqref{pQP} does not imply that $\QP(G)$ admits a polynomial-size extended formulation. 
\end{remark}

\subsection{Proof of Theorem~\ref{th: quad}}
To prove Theorem~\ref{th: quad}, we first present a number of intermediate results that will be used in the proof.
The first proposition is a celebrated result in discrete optimization. Recall that given a hypergraph $G=(V,E)$, the primal treewidth of $G$, denoted by $\tw(G)$ is the treewidth of its intersection graph.

\begin{proposition}[\cite{craHanJau90}]\label{bpoly}
Let $G=(V,E)$ be a hypergraph with the primal treewidth $\tw(G)=\kappa$. Then Problem~\eqref{prob BPO} can be solved by dynamic programming in $O(2^{\kappa}|V|)$ operations. In particular, if $\kappa \in O(\log (|V|+|E|))$, then Problem~\eqref{prob BPO} is solvable in strongly polynomial time.    
\end{proposition}

The next result is concerned with Problem~\eqref{pQP} and essentially states that the variables $x_i$ with $q_{ii} \leq 0$ can be treated as binary variables. We refer to such variables as \emph{hidden binary variables}.

\begin{lemma}\label{lem:binary}
Let $G=(V,E)$ be the interaction graph of Problem~\eqref{pQP} and define $V^+:=\{i\in V: q_{ii} > 0\}$. Then  there exists an optimal solution $x^*$ such that $x^*_i \in \{0,1\}$ for all $i\in V \setminus V^+$.
\end{lemma}

\begin{proof}
Let $x^*$ be a minimizer of Problem~\eqref{pQP}. Consider some $i \in V \setminus V^+$ and fix all other coordinates at their current values.
The resulting univariate function is concave and therefore has an endpoint with objective value no greater than the current value. Since the current point is optimal, equality must hold. Replace $x^*_i$
by that endpoint and proceed to the next coordinate. This preserves optimality throughout and eventually makes every coordinate in $V \setminus V^+$ binary.
\end{proof}

We make use of the following result~\cite{vavasis90} to obtain a polynomial-time face enumeration algorithm to solve certain box-constrained quadratic programs.

\begin{proposition}\cite{vavasis90}\label{vava}
 Let $x^*$ be a minimizer of Problem~\eqref{pQP} satisfying the following properties:
 \medskip
 \begin{itemize}
     \item [(i)] $x^*$ has the maximum number of elements that are equal to $0$ or $1$.
     \item [(ii)] If there are two or more minima satisfying condition~(i), then $x^*$ has the maximum number of elements that are equal to zero. 
 \end{itemize}
 \medskip
 Denote by $\bar x^\top \bar Q \bar x + \bar c^\top \bar x$ the reduced quadratic function obtained by substituting for binary elements of $x^*$ into the objective function of Problem~\eqref{pQP}. Then the matrix $\bar Q$ is positive definite.
\end{proposition}

Thanks to Proposition~\ref{vava} we next obtain a simple algorithm for solving Problem~\eqref{pQP}, which in particular runs in polynomial time when the number of variables is fixed.

\begin{lemma}\label{lem:qp-kappa}
Consider Problem~\eqref{pQP}and let $\kappa \in \{0,\cdots,n-1\}$ be the smallest integer such that every $(n-\kappa)\times(n-\kappa)$ principal submatrix of $Q$ is positive semidefinite. If no such integer exists, set $\kappa := n$. Then Problem~\eqref{pQP} can be solved in time
\begin{equation}\label{tottime}
O\Big((\kappa+1)\Big(\frac{2en}{\kappa}\Big)^\kappa\poly(\L)\Big)
\end{equation}
if $\kappa\geq1$, and in time $\poly(\L)$ if $\kappa=0$.
\end{lemma}

\begin{proof}
For any $B\subseteq[n]$, denote by $Q_{[n]\setminus B}$ the principal submatrix of $Q$ indexed by $[n]\setminus B$. For each subset $B\subseteq[n]$ with $|B|\leq\kappa$ and each $z\in\{0,1\}^B$, consider the problem obtained from Problem~\eqref{pQP} by fixing $x_B=z$. The resulting problem is a box-constrained quadratic optimization problem in the variables $x_{[n]\setminus B}$, whose quadratic matrix is $Q_{[n]\setminus B}$. If $Q_{[n]\setminus B}$ is positive semidefinite, then the resulting problem is a convex quadratic optimization problem and can be solved in time polynomial in $\L$~\cite{TarKha79}.
Consider the following algorithm. For every $B\subseteq[n]$ with $|B|\leq\kappa$ and every $z\in\{0,1\}^B$, check whether $Q_{[n]\setminus B}$ is positive semidefinite. If so, solve the corresponding convex quadratic optimization problem obtained by fixing $x_B=z$. Among all solutions obtained in this way, return one having minimum objective value.

We show that at least one of the convex quadratic optimization problems considered by the algorithm has the same optimal value as Problem~\eqref{pQP}. Let $x^*$ be a global minimizer of Problem~\eqref{pQP} satisfying the properties of Proposition~\ref{vava}, and define
$B^*:=\big\{i\in[n]:x_i^*\in\{0,1\}\big\}$ and
$F^*:=[n]\setminus B^*$.
By Proposition~\ref{vava}, the quadratic matrix of the reduced problem obtained by substituting the binary components of $x^*$ into the objective function is positive definite. Since fixing variables changes only the linear and constant terms involving the remaining variables, this quadratic matrix is precisely $Q_{F^*}$. Hence,
$Q_{F^*}\succ 0$. Suppose first that $|B^*|\leq\kappa$. Consider the choice $B=B^*$ and $z=x^*_{B^*}$. Since $Q_{F^*}\succ0$, the corresponding restricted problem is convex and is therefore solved by the algorithm. Moreover, $x^*$ is feasible for this restricted problem. Since the feasible region of the restricted problem is contained in the feasible region of Problem~\eqref{pQP}, and $x^*$ is a global minimizer of Problem~\eqref{pQP}, it follows that $x^*$ is also a global minimizer of the restricted problem. Thus, this restricted problem has the same optimal value as Problem~\eqref{pQP}.

Now suppose that $|B^*|>\kappa$. Choose any subset $B\subseteq B^*$ with $|B|=\kappa$, and set $z=x_B^*$. By the definition of $\kappa$, every $(n-\kappa)\times(n-\kappa)$ principal submatrix of $Q$ is positive semidefinite. Since $|[n]\setminus B|=n-\kappa$, we have
$Q_{[n]\setminus B}\succeq 0$.
Hence, the problem obtained by fixing $x_B=z$ is convex and is considered by the algorithm. Again, $x^*$ is feasible for this restricted problem, and since $x^*$ is globally optimal for Problem~\eqref{pQP}, it is also optimal for the restricted problem. Therefore, this restricted problem has the same optimal value as Problem~\eqref{pQP}.
Thus, in either case, at least one of the convex quadratic optimization problems solved by the algorithm has the same optimal value as Problem~\eqref{pQP}. On the other hand, the feasible region of every restricted problem considered by the algorithm is contained in $[0,1]^n$, and hence its optimal value cannot be smaller than the optimal value of Problem~\eqref{pQP}. It follows that the best solution returned by the algorithm is a minimizer of Problem~\eqref{pQP}.
It remains to bound the running time. For each $r\in\{0,\ldots,\kappa\}$, there are $\binom{n}{r}$ subsets $B\subseteq[n]$ with $|B|=r$, and for each such subset there are $2^r$ vectors $z\in\{0,1\}^B$. Hence, the total number of restricted problems considered is
$\sum_{r=0}^{\kappa}2^r\binom{n}{r}$.
For $\kappa=0$, this quantity is equal to one. Suppose that $\kappa\geq1$. Using
$
\binom{n}{r}\leq\left(\frac{en}{r}\right)^r
$
for $1\leq r\leq n$, we obtain
$
2^r\binom{n}{r} \leq \left(\frac{2en}{r}\right)^r$.
The function $g(t)= t\log(2en/t)$ is increasing on $[1,n]$, since its derivative is $\log(2n/t)>0$ on this interval. Hence, for every $1\leq r\leq\kappa$, we have
$
\left(\frac{2en}{r}\right)^r
\leq
\left(\frac{2en}{\kappa}\right)^\kappa
$,
implying that,
$$
\sum_{r=0}^{\kappa}2^r\binom{n}{r}
\leq
(\kappa+1)
\left(\frac{2en}{\kappa}\right)^\kappa.
$$
For every pair $(B,z)$, positive semidefiniteness of $Q_{[n] \setminus B}$ can be checked in time polynomial in $\L$. Moreover, whenever $Q_{[n] \setminus B}$ is positive semidefinite, the corresponding box-constrained convex quadratic optimization problem can be solved in time polynomial in $\L$. Therefore, the total running time is
given by~\eqref{tottime}
for $\kappa\geq1$, while for $\kappa=0$ it is $\poly(\L)$, and this completes the proof.
\end{proof}

We are now ready to prove Theorem~\ref{th: quad}.

\begin{proof}[Proof of Theorem~\ref{th: quad}]
Define $V^- := V \setminus V^+$. Let $x_{V^-}$ (resp. $x_{V^+}$) contain the components $x_i$ of $x$ with $i \in V^-$ (resp. $i \in V^+$).  Denote by $f(x)$ the objective function of Problem~\eqref{pQP}. By Lemma~\ref{lem:binary}, to solve Problem~\eqref{pQP}, we can equivalently solve the following problem:
\begin{align*}
    \min \quad & f(x_{V^-}, x_{V^+})\\
    {\rm s.t.} \quad & x_{V^-}\in\{0,1\}^{V^-},\;
x_{V^+}\in[0,1]^{V^+}.
\end{align*}
Let $\C$ denote the set of connected components of $G_{V^+}$. Define
$$
f_{V^-}(x_{V^-}):= \sum_{i\in V^-} q_{ii}x_i^2
+2\sum_{\substack{\{i,j\}\in E:\\ i,j\in V^-}} q_{ij}x_ix_j+
\sum_{i\in V^-} c_i x_i,
$$
and, for each $C\in\C$, define
$$
f_C(x_C,x_{N(C)}):=\sum_{i\in C} q_{ii}x_i^2
+2\sum_{\substack{\{i,j\}\in E:\\ i,j\in C}} {q_{ij}x_ix_j}
+2\sum_{\substack{\{i,j\}\in E\\ i\in C,\ j\in N(C)}} {q_{ij}x_ix_j}+
\sum_{i\in C} {c_i x_i},
$$
where $N(C)$ is defined by~\eqref{nghood}.
Since the sets $C \in \C$ are the connected components of $G_{V^+}$ and $N(C) \subseteq V^-$ for all $C \in \C$, we deduce that:
$$
f(x_{V^-},x_{V^+}) = f_{V^-}(x_{V^-})+\sum_{C\in\C}
{f_C(x_C,x_{N(C)})}.
$$
Since $C \cap C' = \emptyset$ for any $C\neq C' \in \C$, for any fixed $x_{V^-}$ we have:
\begin{equation}\label{e1}
\min_{x_{V^+} \in [0,1]^{V^+}} f(x_{V^-},x_{V^+})
= f_{V^-}(x_{V^-}) +
\sum_{C\in \C} \min_{x_C \in [0,1]^C} f_C(x_C,x_{N(C)}).
\end{equation}
Now, for each $C \in \C$, and for each  $z\in\{0,1\}^{N(C)}$, define
$$
\psi_C(z):= \min_{x_C\in[0,1]^C} f_C(x_C,z).
$$
From~\eqref{e1} we deduce that:
$$
\min_{\substack{x_{V^-} \in \{0,1\}^{V^-}\\x_{V^+} \in [0,1]^{V^+}}} f(x_{V^-},x_{V^+})=
\min_{x_{V^-}\in\{0,1\}^{V^-}}\left(f_{V^-}(x_{V^-})
+\sum_{C\in\C}\psi_C(x_{N(C)})\right).
$$
Thus, once we compute $\psi_C(z)$ for all $z\in\{0,1\}^{N(C)}$ and for all $C \in \C$, the original continuous problem is reduced to a binary optimization problem.  Let us now consider the complexity of computing the function values $\psi_C(z)$, $z \in \{0,1\}^{N(C)}$ for all $C \in \C$. For each fixed $z \in \{0,1\}^{N(C)}$, computing $\psi_C(z)$ amounts to minimizing a quadratic function over $[0,1]^C$. From assumption~(ii) in the statement of the theorem it follows that $\kappa_C\log\big(1+\frac{|C|}{\kappa_C}\big)
\in O(\log |V|)$ for all $C\in\C$, where the expression is defined to be zero when $\kappa_C=0$. If $\kappa_C=0$, then by Lemma~\ref{lem:qp-kappa}, $\psi_C(z)$ can be computed in time $\poly(\L)$.  Now suppose that $\kappa_C\geq1$. By Lemma~\ref{lem:qp-kappa}, for a fixed $z$, $\psi_C(z)$ can be computed in time
$$
O\Big(
(\kappa_C+1)
\left(\frac{2e|C|}{\kappa_C}\right)^{\kappa_C}
\poly(\L)
\Big).
$$
Since $\kappa_C\leq |C|$, we have
$
\kappa_C+1
\leq 2^{\kappa_C}
\leq
\left(1+\frac{|C|}{\kappa_C}\right)^{\kappa_C}.
$
Moreover, setting $t:=\frac{|C|}{\kappa_C}\geq1$, we have
$2et\leq 8t\leq(1+t)^3$,
and hence
$
\left(\frac{2e|C|}{\kappa_C}\right)^{\kappa_C}
\leq
\left(1+\frac{|C|}{\kappa_C}\right)^{3\kappa_C}.
$
It follows that
$
(\kappa_C+1)
\Big(\frac{2e|C|}{\kappa_C}\Big)^{\kappa_C}
\leq
\Big(1+\frac{|C|}{\kappa_C}\Big)^{4\kappa_C}.
$
From assumption~(ii) it follows that
$
\Big(1+\frac{|C|}{\kappa_C}\Big)^{4\kappa_C}
\in |V|^{O(1)}.
$
Hence, for each fixed $z\in\{0,1\}^{N(C)}$, the value $\psi_C(z)$ can be computed in $O(|V|^{O(1)}\poly(\L))$ time.
Moreover, since $N(C)$ is a clique of $\bar G$, from assumption~(i) it follows that $|N(C)| \in O(\log |V|)$. 
Therefore, $\psi_C(z)$ for all $z \in \{0,1\}^{N(C)}$ can be computed in $O(|V|^{O(1)} \poly(\L))$ time. 
Finally, we have $|\C| \leq |V|$. 
Therefore, all such functional computations can be done in time polynomial in the input length $\L$.
Moreover, the algorithm in Lemma~\ref{lem:qp-kappa} computes $\psi_C(z)$ by solving a polynomial number of convex quadratic optimization problems with rational data of encoding length polynomial in $\L$. The optimal value of each such problem can be computed exactly and has encoding length polynomial in $\L$. Consequently, $\psi_C(z)$ can be computed exactly and also has encoding length polynomial in $\L$. Hence, all values defining the functions $\psi_C$ can be represented using a polynomial number of bits.

Finally, let us consider the cost of solving the binary optimization problem: 
\begin{align}\label{reduceproblem}
\min \quad & f_{V^-}(x_{V^-})
+\sum_{C\in\C}\psi_C(x_{N(C)})\\
{\rm s.t.} \quad & x_{V^-}\in\{0,1\}^{V^-}\nonumber.
\end{align}
Since each term $\psi_C(x_{N(C)})$ depends only on the binary variables $x_i$, $i\in N(C)$, it admits a unique multilinear representation in these variables. Moreover, since assumption~(i) implies $|N(C)|\in O(\log |V|)$, this representation can be constructed from the values of $\psi_C$ in polynomial time. It contains at most $2^{|N(C)|}=|V|^{O(1)}$ monomials, and each of its coefficients is a sum of at most $2^{|N(C)|}$ values $\psi_C(z)$. Hence, each coefficient has encoding length polynomial in $\L$. Since $|\C|\leq |V|$, Problem~\eqref{reduceproblem} admits a binary polynomial representation of total encoding length polynomial in $\L$.
Therefore, Problem~\eqref{reduceproblem} can be represented as a binary polynomial optimization problem of polynomial encoding length. Denote by $\bar G$ the graph obtained from $G_{V^-}$ by making each $N(C)$ a clique for all $C \in \C$. It then follows that  the intersection graph of the  hypergraph of the objective function of Problem~\eqref{reduceproblem}, denoted by $H$, is a subgraph of $\bar G$.  From assumption~(i) we deduce that
$
{\rm tw}(H) \leq 
{\rm tw}(\bar G)
\in  O(\log |V|)
$.
Therefore, by Proposition~\ref{bpoly}, Problem~\eqref{reduceproblem} can be solved in time polynomial in $|V| 2^{O(\log |V|)} {\rm poly}(\L) = |V|^{O(1)} \cdot {\rm poly}(\L)$.  Let $x^*_{V^-}$ be an optimal solution of Problem~\eqref{reduceproblem}. For each $C\in\C$, let $x_C^*$ be an optimal solution of the problem defining $\psi_C(x^*_{N(C)})$, that is,
$$
x_C^*\in
\arg\min_{x_C\in[0,1]^C}
f_C(x_C,x^*_{N(C)}).
$$
Since the components in $\C$ are pairwise disjoint,  $x_C^*$, $C\in\C$, together with $x^*_{V^-}$ define a vector $x^*\in[0,1]^V$. By~\eqref{e1} and the definition of $\psi_C$, this vector attains the optimal value of Problem~\eqref{pQP}.
Moreover, by Lemma~\ref{lem:qp-kappa}, each $x_C^*$
can be computed in polynomial time.
Hence, an optimal solution of Problem~\eqref{pQP} can also be recovered in polynomial time and this completes the proof.
\end{proof}

\section{Higher-degree polynomial optimization}
\label{sec: poly}

In this section, we consider a box-constrained polynomial optimization problem of degree $d \geq 3$. 
Unlike the quadratic setting, for $d \geq 3$, the optimal value of a polynomial optimization problem over $[0,1]^n$ may not be rational. In general, the optimal value is an algebraic number, \ie a root of a polynomial with rational coefficients, and  there exists optimal solutions whose coordinates are algebraic numbers.
An exact representation of an algebraic number consists of a square-free defining polynomial together with a rational isolating interval containing the root.
In this section, we obtain a sufficient condition under which Problem~\eqref{polyopt} with $d \geq 3$ can be solved in time polynomial in the input length in the unit-cost algebraic RAM model, which we formally define next.

\paragraph{The computational model.} Throughout this section, running times are measured in the \emph{unit-cost algebraic RAM model}~\cite{Tiwari92}. The machine has ordinary index registers and numerical registers. Numerical registers contain exact rational numbers. Each arithmetic operation $+,\; -,\; \times,\; \div$, each comparison between numerical registers, and each memory operation is assigned unit cost, independently of the bit lengths of the rational numbers involved. The input length $\L$ retains its usual bit-model meaning and denotes the number of bits required to encode the input polynomial. Thus, a running time polynomial in $\L$ in this section means that the number of unit-cost algebraic RAM operations is polynomial in $\L$; the bit lengths of intermediate rational numbers are not included in the running time. Moreover, real algebraic numbers are represented symbolically. A real algebraic number $\gamma$ is represented by a pair $(p,I)$, where $p$ is a univariate square-free polynomial with integer coefficients, and $I$ is a rational interval containing one real root of $p$, namely $\gamma$. Whenever real algebraic representations are manipulated below, we use explicit symbolic algorithms over $\mathbb Q$, and measure their running time by the number of underlying unit-cost rational arithmetic operations and comparisons.
A real algebraic vector is represented by a real univariate representation in the standard sense of~\cite{basu06}. Rational linear combinations of real algebraic numbers are stored formally, by keeping the algebraic numbers separately and recording their rational coefficients. An optimizer of Problem~\eqref{polyopt} is represented by its binary coordinates together with one real univariate representation for the continuous coordinates in each component. The optimal value is represented as a formal rational linear combination of real algebraic numbers.

As in Section~\ref{sec: quad}, we first identify a subset of variables that can be treated as binary variables because there is an optimal solution of Problem~\eqref{polyopt} at which these variables take binary values. As before, we refer to such variables as hidden binary variables.
For a vector $x\in\R^n$ and an index $i\in[n]$, we denote by
$x_{-i}:=(x_1,\dots,x_{i-1},x_{i+1},\dots,x_n)\in\R^{n-1}$
the vector obtained from $x$ by removing its $i$-th coordinate. Accordingly, for $t\in\R$ we write
$(x_{-i},t):=(x_1,\dots,x_{i-1},t,x_{i+1},\dots,x_n)\in\R^n$.
The \emph{standard monomial basis} of the space of polynomials in $n$ variables of degree at most $d$ is the set $\{x^\alpha : \alpha \in \mathbb{Z}_{\geq 0}^n,\, |\alpha|\le d\}$; a polynomial is said to be written in this basis if it is expressed as a linear combination of these monomials.
The following lemma provides easily verifiable sufficient conditions for identifying hidden binary variables:

\begin{lemma}\label{cor:simple-cases}
Let $f(x)$, $x \in \R^n$ be a polynomial and fix some $i\in[n]$.
Then we have the following:
\medskip
\begin{enumerate}
\item Write $f(x)$ as a polynomial in $x_i$; that is,
$f(x)=\sum_{m=0}^d a_m(x_{-i})\,x_i^m$,
where each $a_m(x_{-i})$ is a polynomial in the remaining variables $x_{-i}$.
Assume that for every $m\ge 2$, the polynomial $a_m(x_{-i})$ has only nonpositive coefficients when written in the standard monomial basis.
Then there exists a minimizer $x^*$ of $f(x)$ over $[0,1]^n$ such that
$x_i^*\in\{0,1\}$.

\item Define
$$
\Delta_i(x):=f(x)-(1-x_i)f(x_{-i},0)-x_i f(x_{-i},1),
$$
and write
$\Delta_i(x)=x_i(1-x_i)H_i(x)$.
Assume that $H_i(x)$ has only nonnegative coefficients when written in the standard monomial basis.
Then there exists a minimizer $x^*$ of $f(x)$ over $[0,1]^n$ such that $x_i^*\in\{0,1\}$.
\end{enumerate}
\end{lemma}

\begin{proof}
First consider part~(1). Fix $x_{-i}\in[0,1]^{n-1}$ and consider
$$
\phi(t)=f(x_1,\dots,x_{i-1},t,x_{i+1},\dots,x_n)
=\sum_{m=0}^d a_m(x_{-i})\,t^m.
$$
Since each monomial in the variables $x_{-i}$ is nonnegative on $[0,1]^{n-1}$, and since each coefficient of $a_m(x_{-i})$ is nonpositive for every $m\ge 2$, it follows that $\phi(t)$ is concave on $[0,1]$. Therefore, there exists a minimizer of $f(x)$, $x \in [0,1]^n$ at $x^*_i \in \{0,1\}$.
Next, consider part~(2). Since every monomial is nonnegative on $[0,1]^n$ and $H_i$ has only nonnegative coefficients in the standard monomial basis, we have
$H_i(x)\ge 0$ for all  $x\in[0,1]^n$. Therefore $\Delta_i(x) \geq 0$ for all  $x\in[0,1]^n$.
Thus, the slice of $f(x)$ in $x_i$ lies above its chord for every fixing of the other variables, implying that there exists a minimizer $x^*$ of $f(x)$ with $x^*_i \in \{0,1\}$. 
\end{proof}

\begin{remark}\label{ex:factorization-more-general}
In Lemma~\ref{cor:simple-cases}, while condition~1 is simpler to verify,  condition~2 is strictly more general. To see this, consider $f(x_1,x_2)=x_1(1-x_1)(x_1+x_2)$.
Expanding in powers of $x_1$, we obtain
$f(x_1,x_2)= -x_1^3 +(1-x_2)x_1^2 + x_1x_2$.
The coefficient of $x_1^2$ is the polynomial $1-x_2$, which has a positive coefficient in the standard monomial basis. Therefore, condition~1 of Lemma~\ref{cor:simple-cases} does not apply.
On the other hand,
$f(0,x_2)=f(1,x_2)=0$, implying that
$\Delta_1(x_1,x_2)=f(x_1,x_2)=x_1(1-x_1)H_1(x_1,x_2)$,
with $H_1(x_1,x_2)=x_1+x_2$.
Since $H_1$ has only nonnegative coefficients in the standard monomial basis, condition~2 of Lemma~\ref{cor:simple-cases} applies. Hence there exists a minimizer of $f$ with $x_1\in\{0,1\}$.
For $d=2$, conditions~(1) and~(2) of Lemma~\ref{cor:simple-cases} are equivalent and both reduce to the condition of Lemma~\ref{lem:binary}; \ie $q_{ii}\le 0$. To see this consider $f(x)=x^\top Q x + c^\top x$. Viewing $f$ as a polynomial in $x_i$, we obtain
$$
f(x)=q_{ii}x_i^2+\Big(2\sum_{j\neq i}q_{ij}x_j+c_i\Big)x_i+\text{terms independent of }x_i.
$$
Thus $a_2(x_{-i})=q_{ii}$, and condition~(1) is equivalent to $q_{ii}\le 0$.
Moreover, a direct computation gives
$\Delta_i(x)=f(x)-(1-x_i)f(x_{-i},0)-x_i f(x_{-i},1)
=-q_{ii}\,x_i(1-x_i)$.
Hence in condition~(2) we have $H_i(x)=-q_{ii}$, and requiring $H_i$ to have nonnegative coefficients is again equivalent to $q_{ii}\le 0$.
\end{remark}

We should remark that a more general sufficient condition than those of Lemma~\ref{cor:simple-cases} 
is $H_i(x)\geq 0$ for all $x\in[0,1]^n$. 
However, this condition cannot be verified in polynomial time, in general.
Let $G=(V,E)$ be a hypergraph and let $S \subseteq V$. We define the \emph{subhypergraph} of $G$ induced by $S$ as $G_{S} = (S, E_{S})$ where $E_{S} := \{e \cap S : e \in E,\ |e \cap S| \ge 2 \}$. A hypergraph $G=(V,E)$ is \emph{connected} if for every $u,v\in V$ there exists a sequence of edges $e_1,\dots,e_k\in E$
such that $u\in e_1$, $v\in e_k$, $e_j \cap e_{j+1} \neq \emptyset$ for all $j \in [k-1]$. The \emph{connected components} of a hypergraph are the maximal subsets of nodes that induce connected subhypergraphs.
Recall that the incidence graph $I(G)$ of a hypergraph $G$ is the bipartite graph with node set $V\cup E$,
in which $v\in V$ is adjacent to $e\in E$ if and only if $v\in e$.
We define the \emph{incidence treewidth} of $G$, denoted by
$\itw(G)$, as the treewidth of $I(G)$. 
\medskip

We are now ready to state the main result of this section.

\begin{theorem}\label{th: polyopt}
Consider Problem~\eqref{polyopt} with the interaction hypergraph $G=(V,E)$. Suppose that $d \in O(\L^c)$ for some fixed constant $c > 0$, where $\L$ denotes the number of bits required to store the objective function. Let $V^-$ be the set of hidden binary variables as defined by Lemma~\ref{cor:simple-cases}, and let $V^+:= V \setminus V^-$. Denote by $\C$ the set of connected components of $G_{V^+}$. For each component $C \in \C$, define its neighborhood:
$$
N(C) := \big\{ u \in V \setminus C : \exists\, e \in E \text{ such that } e \cap C \neq \emptyset,\ u \in e \big\}.
$$
Denote by $\bar G$ the hypergraph obtained from $G_{V^-}$ by adding for every $C \in \C$, every edge corresponding to a subset of $N(C)$ of cardinality at least two; \ie $\bar G=(V^-, \bar E)$, where 
$$\bar E= E_{V^-} \cup \bigcup_{C \in \C}\{S \subseteq N(C): |S| \geq 2\},
$$ 
and where $E_{V^-}$ denotes the edge set of $G_{V^-}$.
Suppose that the following conditions are satisfied:
\medskip
\begin{itemize}[leftmargin=1.0cm]
\item [(i)] $\itw(\bar G)\in O(\log |V|)$;
\item [(ii)] $|C|\in O(1)$ for all $C\in \C$;
\end{itemize}
\medskip
Then Problem~\eqref{polyopt} can be solved in time polynomial in $\L$ in the unit-cost algebraic RAM model. Moreover, the algorithm returns an exact symbolic representation of an optimal solution and of the optimal value.
\end{theorem}

\begin{remark}
The structural objects needed to apply Theorem~\ref{th: polyopt} can be constructed in time polynomial in $\L$ for instances satisfying its assumptions.  Indeed, the interaction hypergraph $G=(V,E)$ can be constructed in polynomial time from the monomial representation of the polynomial. 
The set $V^-$ can be identified by checking the
sufficient conditions of Lemma~\ref{cor:simple-cases} for each variable. Condition~(1) can be checked directly from the coefficients of the input polynomial. For condition~(2), writing
$f(x)=\sum_{m=0}^d a_m(x_{-i})x_i^m$,
we have
$H_i(x)=-\sum_{m=2}^d a_m(x_{-i})\left(1+x_i+\cdots+x_i^{m-2}\right)$.
Thus, each input monomial generates at most $d-1$ monomials in the explicit representation of $H_i$. Since $d$ is polynomially bounded in $\L$, $H_i$ can be constructed explicitly and its coefficient signs can
be checked in time polynomial in $\L$.
The connected components of $G_{V^+}$, as well as $N(C)$ for each component $C$, can be computed by performing a graph traversal on the incidence graph of $G$ in time linear in the size of the incidence representation, namely $O\!\left(|V|+|E|+\sum_{e\in E}|e|\right)$. 
Verifying condition~(ii), namely that $|C|\in O(1)$ for all $C\in\C$, is then immediate.
We next consider the construction of $\bar G$. For every $C\in\C$, the hypergraph $\bar G$ contains all subsets of cardinality two of $N(C)$. That is, $\bar G$ contains a complete graph on $N(C)$, implying 
$\itw(\bar G)\ge |N(C)|-1$. Hence, assumption~(i) implies
$|N(C)|\in O(\log |V|)$ for all $C\in\C$.
It follows that the number of edges added to form the complete hypergraph on $N(C)$ is at most
$2^{|N(C)|}=|V|^{O(1)}$.
Since $|\C|\le |V|$, the hypergraph $\bar G$ has polynomial size and can therefore be constructed explicitly in time polynomial in $\L$.
Finally, although computing treewidth exactly is $\NP$-hard, there exists an algorithm that, given a graph with $N$ nodes and an integer $k$, runs in time $2^{O(k)}N$ and either certifies that its treewidth is greater than $k$ or constructs a tree decomposition of width $O(k)$~\cite{BodEtAl16}. Applying this algorithm to the incidence graph of $\bar G$, assumption~(i) implies that a tree decomposition of $I(\bar G)$ of width $O(\log |V|)$ can be constructed in time polynomial in $\L$.
\end{remark}

To prove Theorem~\ref{th: polyopt} we make use of the following three results. The first result serves as the strongest sufficient condition for polynomial-time solvability of binary polynomial optimization. In the following, by $\poly(|V|,|E|)$, we imply a polynomial in $|V|, |E|$.

\begin{proposition}[\cite{CapDelDiG23}]\label{binaryitw}
There is a strongly polynomial time algorithm to solve Problem~\eqref{prob BPO} if the interaction hypergraph $G$ is $\beta$-acyclic or if $\itw(G) \in O(\log(|V|+|E|))$. Moreover, an optimal solution of Problem~\eqref{prob BPO} can be found using $ 2^{O(\itw(G))}\poly(|V|,|E|)$ arithmetic operations and comparisons involving the objective function coefficients.   
\end{proposition}

Now let us turn our attention to a continuous polynomial optimization problem. The problem of minimizing a polynomial over a compact semialgebraic set can be reduced to deciding the feasibility of semialgebraic systems. It is a classical result that such problems can be solved in time polynomial in the encoding length of the input when the number of variables is fixed (see~\cite{renegar92} and~\cite{basu06}). More precisely, from Algorithm~14.1, Algorithm~14.9, and Theorem~13.11 of~\cite{basu06}, we obtain:

\begin{proposition}[\cite{basu06}]
\label{cor:log-dim-poly-opt}
Given a polynomial $f(x)$, $x\in\R^n$, of degree at most $d$ and with rational coefficients, the problem of minimizing $f$ over $[0,1]^n$ can be solved exactly using $(nd)^{O(n)}$ arithmetic operations over
$\mathbb Q$. The algorithm returns an exact real algebraic representation of the optimal value together with a real univariate representation of an optimal solution. For fixed $n$, the degrees and encoding lengths of these representations are polynomially bounded in $d$ and $\L$.
\end{proposition}

Recall that we assume $d$ is polynomially bounded in $\L$. Therefore, for fixed $n$,
Proposition~\ref{cor:log-dim-poly-opt} implies that
Problem~\eqref{polyopt} can be solved using a number of unit-cost algebraic RAM operations polynomial in $\L$.
To prove Theorem~\ref{th: polyopt}, we first eliminate the
continuous components which yields a binary polynomial optimization
problem whose coefficients are rational linear combinations
of algebraic numbers. Solving this reduced problem requires
exact comparisons between such combinations. The following
lemma, which is a direct
consequence of the separation bound of~\cite{BurFunMehSch09} and root refinement results of~
\cite{Pan13}, shows that such comparisons can be performed using
polynomially many rational arithmetic operations and comparisons
in the unit-cost algebraic RAM model. 

\begin{lemma}\label{lem:algebraic-sign}
For each $j\in[M]$, let the real algebraic number $\gamma_j$ be specified by a
square-free polynomial $p_j\in\mathbb{Z}[T]$, given by its full coefficient list, and a rational
isolating interval $I_j$.
Let $q_0,\ldots,q_M\in\mathbb{Q}$, and let $\ell$ denote the
total binary encoding length of the data. Then the sign of
$\theta=q_0+\sum_{j=1}^{M}q_j\gamma_j$
can be determined using $\operatorname{poly}(\ell)$ rational
arithmetic operations and comparisons in the unit-cost algebraic RAM model.
\end{lemma}

\begin{proof}
Applying Theorem~1 of~\cite{BurFunMehSch09} to 
$\theta$ yields an integer $B\geq2$, computable using
polynomially many rational operations, such that
\begin{equation}\label{sepbound}
\theta=0\quad\text{or}\quad |\theta|\geq\delta:=2^{-B},\;
{\rm such \; that}\; \log (B)\in O({\rm poly}(\ell)).
\end{equation} 
The bound on $\log(B)$ follows by specializing Theorem~1
of~\cite{BurFunMehSch09} to a rational linear combination of algebraic numbers; namely, $B$ is at most a polynomial
in $\ell$ times the product of the defining polynomial degrees $\prod_{j\in [M]}\deg p_j$, whose
logarithm is polynomially bounded in $\ell$.
The number $\delta$ can be constructed by repeated squaring
using $O(\log B)$ operations.
It then suffices to approximate $\theta$ with error at
most $\delta/4$ and compare the approximation with
$\pm\delta/2$.

Set $C=1+\sum_{j=1}^{M}|q_j|$, and choose an integer $s\geq0$
such that $C\leq2^s$. Since the rational coefficients $q_j$ have
encoding length polynomial in $\ell$, we can choose
$s\in\operatorname{poly}(\ell)$.
Let $P$ be an integer precision parameter; that is, a rational approximation $\widetilde\gamma_j$ of $\gamma_j$ with precision $P$ has absolute error at most $2^{-P}$ for every $j \in [M]$.
Choose $P= B+s+2$.
Define the approximate sum
$\widetilde\theta
:=q_0+\sum_{j=1}^{M}q_j\widetilde\gamma_j$.
Since $\delta=2^{-B}$, its error satisfies
$$
|\theta-\widetilde\theta| \leq
\sum_{j=1}^{M}|q_j|\,
|\gamma_j-\widetilde\gamma_j|\leq C2^{-P}
\leq 2^s2^{-(B+s+2)}
=\frac{\delta}{4}.
$$
Thus, if $\widetilde\theta>\delta/2$, then
$\theta\geq\widetilde\theta-\delta/4>0$.
Similarly, if $\widetilde\theta<-\delta/2$, then
$\theta\leq\widetilde\theta+\delta/4<0$.
Otherwise, $|\widetilde\theta|\leq\delta/2$, and hence
$|\theta|
\leq|\widetilde\theta|+|\theta-\widetilde\theta|
\leq\frac{3\delta}{4}
<\delta$.
By~\eqref{sepbound}, this implies
$\theta=0$. Therefore, rational approximations $\widetilde\gamma_j$ suffice to determine the sign of $\theta$ exactly.
We obtain the approximations using root refinement method of~\cite{Pan13}.
For $m\in\mathbb{R}$ and $\rho>0$, let
$D(m,\rho):=\{z\in\mathbb{C}:|z-m|\leq\rho\}$.
Theorem~9 of~\cite{Pan13} implies that if
$D(m,\rho)$ and $D(m,5d^2\rho)$
contain exactly the same simple root $\gamma$ of a degree-$d$
polynomial, then Newton refinement starting at $m$ converges
quadratically to $\gamma$. Moreover, Remark~11 of~\cite{Pan13} gives
$O(\log(L+2))$ exact-rational refinement steps for a refinement
factor $2^{-L}$. Including polynomial evaluations, this requires
$O(d\log(L+2))$ rational operations.

We now construct an interval satisfying the above hypothesis for
each $p_j$. Write $d_j=\deg p_j$. If $d_j=1$, compute
$\gamma_j$ directly by rational division. Otherwise,
Proposition~1 of~\cite{Pan13} implies a computable rational
lower bound $\sigma_j\in(0,1]$ on the distance between distinct
complex roots of $p_j$, with
$\log(1/\sigma_j)\in\poly(\ell)$.
Bisect the supplied interval $I_j$, retaining the half on
which $p_j$ changes sign, until the resulting interval $J_j$
has width
$w_j\leq\frac{\sigma_j}{10d_j^2}$.
If a midpoint is a root, return it exactly. This initialization
uses polynomially many rational operations, independently of
$P$, because the initial endpoints have polynomial encoding
length and $\log(10d_j^2/\sigma_j)\in\poly(\ell)$.
Let $m_j$ be the midpoint of $J_j$ and set $\rho_j=w_j/2$.
Then $|\gamma_j-m_j|\leq\rho_j$, whereas every other root
$\beta$ satisfies
$|\beta-m_j|
\geq\sigma_j-\rho_j
>5d_j^2\rho_j$.
Thus $D(m_j,\rho_j)$ and $D(m_j,5d_j^2\rho_j)$ contain
exactly the root $\gamma_j$. The root is simple because
$p_j$ is square-free. The hypothesis of Theorem~9 is
therefore satisfied.
Our construction also gives $w_j\leq1$. Hence refinement
by a factor $2^{-(P+1)}$ suffices to produce a rational
approximation $\widetilde\gamma_j$ satisfying
$|\gamma_j-\widetilde\gamma_j|\leq2^{-P}$.
Applying the exact-rational refinement stage with $L=P+1$
therefore requires $O(d_j\log(P+2))$ rational operations.
Since $\log(P+2)\in\poly(\ell)$, the total cost of
initialization and refinement is
\[
\poly(\ell)+
O\Big(\sum_{j=1}^{M}d_j\log(P+2)\Big)
=\poly(\ell).
\]
\end{proof}

We are now ready to prove Theorem~\ref{th: polyopt}.

\begin{proof}[Proof of Theorem~\ref{th: polyopt}]
For any $U \subseteq V$, denote by $x_U$ the vector containing all components $x_i$ of $x$ with $i \in U$.
From Lemma~\ref{cor:simple-cases} it follows that there exists a minimizer of Problem~\eqref{polyopt} at which all variables indexed by $V^-$ are binary.  Consequently,
\begin{equation}\label{eq:mixed-polyopt}
\min_{x\in[0,1]^V} f(x)
=
\min_{\substack{x_{V^-}\in\{0,1\}^{V^-}\\
x_{V^+}\in[0,1]^{V^+}}}
f(x_{V^-},x_{V^+}).
\end{equation}
We next decompose $f(x)$ according to the connected components of $G_{V^+}$. First observe that
$N(C)\subseteq V^-$ for every $C\in\C$. Indeed, suppose that $u\in N(C)\cap V^+$. Then there exists an edge $e\in E$ containing $u$ and a node of $C$. Hence, $e\cap V^+$ is an edge of $G_{V^+}$ containing $u$ and a node of $C$, contradicting the fact that $C$ is a connected component of $G_{V^+}$. Now consider any monomial $x^\alpha$ whose support intersects $V^+$. All nodes in $\supp(\alpha)\cap V^+$ belong to a unique component $C\in\C$. Indeed, this is immediate if $|\supp(\alpha)\cap V^+|=1$, while if $|\supp(\alpha)\cap V^+|\geq2$, then $\supp(\alpha)\cap V^+$ is an edge of $G_{V^+}$ and hence all its nodes belong to the same connected component. Moreover, every node in $\supp(\alpha)\setminus C$ belongs to $N(C)$. Thus, every monomial whose support intersects $V^+$ depends on the variables of exactly one component $C$ and possibly on variables in $N(C)$.
Let $f_{V^-}(x_{V^-})$ be the sum of all monomials of $f$ whose support is contained in $V^-$, and, for each $C\in\C$, let $f_C(x_C,x_{N(C)})$ be the sum of all monomials whose support intersects $C$. By the preceding observation, these collections of monomials form a partition of the monomials of $f$, and therefore
\begin{equation}\label{polyrep}
f(x)=f_{V^-}(x_{V^-})
+\sum_{C\in\C}f_C(x_C,x_{N(C)}).
\end{equation} 
Since $C \cap C' = \emptyset$ for all $C \neq C' \in \C$, from~\eqref{polyrep} it follows that
$$
\min_{x_{V^+} \in [0,1]^{V^+}} f(x_{V^-},x_{V^+})
=
f_{V^-}(x_{V^-})
+
\sum_{C \in \C} \min_{x_C \in [0,1]^C} f_C(x_C, x_{N(C)}).
$$
For each $C \in \C$ and each $z \in \{0,1\}^{N(C)}$, define
$$
\psi_C(z) := \min_{x_C \in [0,1]^C} f_C(x_C, z).
$$
It follows from~\eqref{eq:mixed-polyopt} that
\begin{equation}\label{use3}
\min_{x \in [0,1]^V} f(x)
=
\min_{x_{V^-} \in \{0,1\}^{V^-}}\Big(f_{V^-}(x_{V^-})
+
\sum_{C \in \C} \psi_C(x_{N(C)})\Big).
\end{equation}
We now examine the cost of computing the functions $\psi_C$. Fix $C\in\C$ and $z\in\{0,1\}^{N(C)}$. By assumption~(ii), $|C| \in O(1)$. Hence, by Proposition~\ref{cor:log-dim-poly-opt} and the assumption that $d$ is polynomially bounded in $\L$, the value
$\psi_C(z)$ and a corresponding minimizer can be computed exactly using a number of unit-cost rational arithmetic operations polynomial in $\L$. Moreover, $\psi_C(z)$ can be represented by a square-free polynomial with integer coefficients together with a rational isolating
interval, and a corresponding minimizer can be represented by a real univariate representation. The degrees and encoding lengths of these representations are polynomially bounded in $\L$. By assumption~(i), and the definition of the hypergraph $\bar G$, we have $|N(C)|-1\leq \itw(\bar G) \in O(\log |V|)$, implying that $2^{|N(C)|}=|V|^{O(1)}$. Since $|\C|\leq |V|$, all values $\psi_C(z)$, for $C\in\C$ and $z\in\{0,1\}^{N(C)}$, together with corresponding minimizers, can be computed using a number of unit-cost algebraic RAM operations polynomial in $\L$.
Let $\gamma_1,\ldots,\gamma_M$
denote all the algebraic values obtained in this way. Then
$M=\sum_{C\in\C}2^{|N(C)|}=|V|^{O(1)}$,
and every $\gamma_j$ has a defining polynomial and isolating interval of degree and encoding length polynomially bounded in $\L$. Each function $\psi_C$ admits a unique multilinear representation $\psi_C(z)=\sum_{S\subseteq N(C)}
a_{C,S}\prod_{i\in S}z_i$,
where, by Mobius inversion,
$a_{C,S}=\sum_{T\subseteq S}(-1)^{|S|-|T|}
\psi_C(\mathbf 1_T)$,
and $\mathbf 1_T$ denotes the binary vector whose entries indexed by $T$ are equal to one and whose remaining entries are equal to zero. We store every coefficient $a_{C,S}$ as a formal integer linear combination of $\gamma_1,\ldots,\gamma_M$. Since
$\sum_{S\subseteq N(C)}2^{|S|}=3^{|N(C)|}=|V|^{O(1)}$,
all these representations can be constructed in polynomial time and have polynomial total encoding length. Therefore, the problem on the right-hand side of~\eqref{use3} can be represented as a binary polynomial optimization problem whose coefficients are formal rational linear combinations of
$1,\gamma_1,\ldots,\gamma_M$.

Let $G_{\rm red}$ denote the hypergraph containing the supports of cardinality at least two that occur in this binary polynomial representation.
Denote by $\bar G$ the hypergraph obtained from $G$ by removing nodes in $C$ and adding the complete hypergraph on $N(C)$ for all $C \in \C$. It then follows that $G_{\rm red}\subseteq\bar G$,
and hence $\itw(G_{\rm red})\leq \itw(\bar G)$.
Furthermore, $G_{\rm red}$ has a polynomial number of edges, since each function $\psi_C$ contributes at most $2^{|N(C)|}$ candidate monomials.
Therefore, it remains to solve a binary polynomial optimization problem over $\bar G$ with the incidence treewidth $\itw(\bar G)$.

Since $G_{\rm red}$ has polynomial size and
$\itw(G_{\rm red})\in O(\log |V|)$, the algorithm underlying
Proposition~\ref{binaryitw}, applied to the reduced problem
on the right-hand side of~\eqref{use3}, performs polynomially
many additions and comparisons. Its correctness extends
to real algebraic coefficients provided these operations
are implemented exactly.
Every cost comparison
performed by the algorithm of Proposition~\ref{binaryitw} is a sum of a subset of the reduced
objective coefficients (see the proof of Theorem~4 of~\cite{CapDelDiG23}).
As established above, these coefficients have 
representations as rational linear combinations of
$1,\gamma_1,\ldots,\gamma_M$ of polynomial total encoding
length. Therefore, every such sum, and the difference
of any two such sums, also has a representation
of encoding length polynomial in $\L$.
We perform additions coordinate-wise on the rational
coefficient vectors and use Lemma~\ref{lem:algebraic-sign}
to determine the sign of each difference. Including the
representations of the $\gamma_j$'s, the total input length
of every application of this lemma satisfies
$\ell\in\poly(\L)$. Thus every addition and comparison
requires polynomially many unit-cost rational operations.
Consequently, the reduced problem can be solved exactly
in time polynomial in $\L$ in the stated unit-cost
algebraic RAM model.

Let $x_{V^-}^*$ be an optimal binary solution obtained in this manner. For each $C\in\C$, retrieve the minimizer previously computed for $\psi_C(x_{N(C)}^*)$.
Since the components in $\C$ are pairwise disjoint, these minimizers, together with $x_{V^-}^*$, define a vector $x^*\in[0,1]^V$. By~\eqref{use3}, this vector is a minimizer of Problem~\eqref{polyopt}.
The binary coordinates of $x^*$ are represented explicitly, and the continuous coordinates are represented by one real univariate representation for each component $C\in\C$. The optimal value is represented by the exact formal sum
$f_{V^-}(x_{V^-}^*)
+\sum_{C\in\C}\psi_C(x_{N(C)}^*)$.
All these representations have total encoding length polynomial in $\L$. This completes the proof.
\end{proof}

Let us conclude by commenting on some of the assumptions and limitations of Theorem~\ref{th: polyopt}.
First, the condition on the size of the continuous components is more restrictive than in the quadratic case of Theorem~\ref{th: quad}. In the quadratic setting, the complexity of each continuous component is controlled by the parameter $\kappa_C$ through assumption~(ii).
Consequently, when $\kappa_C=\Theta(|C|)$, this condition requires $|C|\in O(\log |V|)$, while if $\kappa_C=\Theta(1)$, there is no restriction on the size of $C$. In contrast, Theorem~\ref{th: polyopt} requires $|C|\in O(1)$ for all $C\in\C$.  This difference arises from the continuous subproblems. The quadratic result exploits the special structure of quadratic functions to reduce the optimization problem to a polynomial number of convex quadratic programs. For general polynomial optimization, the algorithm used in Proposition~\ref{cor:log-dim-poly-opt} requires $(kd)^{O(k)}\poly(\L)$ arithmetic operations in dimension $k$. Since $d$ is allowed to be polynomial in $\L$, this bound yields polynomial-time solvability only when $k$ is bounded by a constant. Thus, the higher-degree result requires a substantially stronger restriction on the continuous components.
Second, unlike Theorem~\ref{th: quad}, which gives polynomial-time solvability in the standard bit model, Theorem~\ref{th: polyopt} is stated in the unit-cost algebraic RAM model. This distinction arises because eliminating a continuous component in the higher-degree case may produce algebraic, rather than rational, optimal values. The unit-cost algebraic RAM model allows these values to be represented symbolically and compared exactly in polynomially many arithmetic operations.

\bigskip
\noindent
\textbf{Acknowledgments:}  The author thanks the  anonymous referee whose suggestions significantly improved the quality of this manuscript.

\begin{footnotesize}
\bibliographystyle{plain}
\bibliography{biblio}
\end{footnotesize}

\end{document}